\documentclass[a4paper]{amsart}
\usepackage{wrapfig}
\usepackage[dvips]{graphicx}
\usepackage{amsmath,amsthm,amssymb,amscd}
\usepackage{mathrsfs}
\usepackage[all]{xy}
\theoremstyle{definition}
\newtheorem{thm}{Theorem}[section]
\newtheorem{Def}[thm]{Definition}
\newtheorem{pro}[thm]{Proposition}
\newtheorem{cor}[thm]{Corollary}
\newtheorem{lem}[thm]{Lemma}

\newtheorem*{mainthm}{Theorem A}
\newtheorem*{mainthm2}{Theorem B}
\theoremstyle{definition}

\begin{document}
\title[Strongly outer actions on the Razak-Jacelon algebra]{
Strongly outer actions of certain torsion-free amenable groups on the Razak-Jacelon algebra}
\author{Norio Nawata}
\address{Department of Pure and Applied Mathematics, Graduate School of Information Science 
and Technology, Osaka University, Yamadaoka 1-5, Suita, Osaka 565-0871, Japan}
\email{nawata@ist.osaka-u.ac.jp}
\thanks{This work was supported by JSPS KAKENHI Grant Number 20K03630}

\begin{abstract}
Let $\mathfrak{C}$ be the smallest class of countable discrete groups with the following properties:
(i) $\mathfrak{C}$ contains the trivial group, (ii) $\mathfrak{C}$ is closed under isomorphisms, countable 
increasing unions and extensions by $\mathbb{Z}$. 
Note that $\mathfrak{C}$ contains all countable discrete torsion-free abelian groups and 
poly-$\mathbb{Z}$ groups. Also, $\mathfrak{C}$ is a subclass of the class of 
countable discrete torsion-free elementary amenable groups. 
In this paper, we show that if $\Gamma\in \mathfrak{C}$, 
then all strongly outer actions of 
$\Gamma$ on the Razak-Jacelon algebra $\mathcal{W}$ are cocycle conjugate to each other. 
This can be regarded as an analogous result of Szab\'o's result for strongly self-absorbing 
C$^*$-algebras. 
\end{abstract}
\maketitle
\section{Introduction}

Let $\mathcal{W}$ be the Razak-Jacelon algebra studied in \cite{J} (see also \cite{Raz}). 
By classification results in \cite{CE} and \cite{EGLN} (see also \cite{Na4}), 
$\mathcal{W}$ is the unique simple separable nuclear 
monotracial $\mathcal{Z}$-stable C$^*$-algebra that is $KK$-equivalent to $\{0\}$. 
Also, $\mathcal{W}$ is regarded as a stably finite analog of the Cuntz algebra $\mathcal{O}_2$.
More generally, we can consider that $\mathcal{W}$ is a non-unital analog of strongly self-absorbing 
C$^*$-algebras. (Note that every strongly self-absorbing C$^*$-algebra is unital by definition.) 
In this paper, we study group actions on $\mathcal{W}$ and show an analogous result of 
 Szab\'o's result in \cite{Sza7} for group actions on strongly self-absorbing 
C$^*$-algebras (see also \cite{IM1}, \cite{IM2}, \cite{IM3}, \cite{M1}, \cite{M2}, 
\cite{MS1}, \cite{MS2} \cite{Sza3} and \cite{Sza4} for pioneering works).  
We refer the reader to \cite{I} for the importance and some difficulties of studying group actions 
on C$^*$-algebras.  
Gabe and Szab\'o classified outer actions of countable discrete amenable groups on 
Kirchberg algebras up to cocycle conjugacy in \cite{GS}. 
In their classification, $\mathcal{O}_2$ and $\mathcal{O}_{\infty}$ play central roles. 
Hence it is natural to expect that $\mathcal{W}$ plays a central role in the classification theory of 
group actions on ``classifiable'' stably finite (at least stably projectionless) C$^*$-algebras.  

Let $\mathfrak{C}$ be the smallest class of countable discrete groups with the following properties:
(i) $\mathfrak{C}$ contains the trivial group, (ii) $\mathfrak{C}$ is closed under isomorphisms, countable 
increasing unions and extensions by $\mathbb{Z}$. (We say that $\Gamma$ is an extension by 
$\mathbb{Z}$ if there exists an exact sequence $1\to H \to \Gamma \to \mathbb{Z}\to 1$.) 
Note that $\mathfrak{C}$ is the same class as in \cite[Definition B]{Sza7}. 
It is easy to see that $\mathfrak{C}$ contains all countable discrete torsion-free abelian groups and 
poly-$\mathbb{Z}$ groups, and $\mathfrak{C}$ is a subclass of the class of countable discrete 
torsion-free elementary amenable groups. 
Szab\'o showed that if $\Gamma\in \mathfrak{C}$ and $\mathcal{D}$ is a strongly self-absorbing 
C$^*$-algebra, then there exists a unique strongly outer action of $\Gamma$ on $\mathcal{D}$ up to 
cocycle conjugacy (\cite[Corollary 3.4]{Sza7}). In this paper, we show an analogous result of this result. 
Indeed, the main theorem in this paper is the following theorem. 

\begin{mainthm}
(Theorem \ref{thm:main}.) \ \\
Let $\Gamma$ be a countable discrete group in $\mathfrak{C}$, and let $\alpha$ be a strongly outer 
action of $\Gamma$ on $\mathcal{W}$. Then $\alpha$ is cocycle conjugate to 
$\mu^{\Gamma}\otimes \mathrm{id}_{\mathcal{W}}$ on $M_{2^{\infty}}\otimes\mathcal{W}$ where 
$\mu^{\Gamma}$ is the Bernoulli shift action of 
$\Gamma$ on $\bigotimes_{g\in \Gamma}M_{2^{\infty}}\cong M_{2^{\infty}}$.
\end{mainthm}

We say that an action $\alpha$ on $\mathcal{W}$ is \textit{$\mathcal{W}$-absorbing} if 
there exists a simple separable nuclear monotracial C$^*$-algebra $A$ and an action  $\beta$ 
on $A$ such that $\alpha$ is cocycle conjugate to $\beta\otimes\mathrm{id}_{\mathcal{W}}$ 
on $A\otimes\mathcal{W}$. 
The proof of the main theorem above is based on a characterization in \cite{Na5} of strongly outer 
$\mathcal{W}$-absorbing actions of countable discrete amenable groups. 
Actually, we use the following theorem that is a slight variant of \cite[Theorem 8.1]{Na5}. 
Note that $F(\mathcal{W})$ is Kirchberg's central sequence C$^*$-algebra of 
$\mathcal{W}$. Furthermore, $F(\mathcal{W})^{\alpha}$ is the fixed point algebra for the action on 
$F(\mathcal{W})$ induced by an action $\alpha$ on $\mathcal{W}$. 
Let $\mathrm{Sp}(x)$ denote the spectrum of $x$. 

\begin{mainthm2}
(Theorem \ref{thm:8.1}.) \ \\
Let $\alpha$ be a strongly outer action of a countable discrete amenable group $\Gamma$ on 
$\mathcal{W}$. Then $\alpha$ is cocycle conjugate to 
$\mu^{\Gamma}\otimes\mathrm{id}_{\mathcal{W}}$ on $M_{2^{\infty}}\otimes \mathcal{W}$
if and only if $\alpha$ satisfies the following properties: \ \\
(i) there exists a unital  $*$-homomorphism from $M_{2}(\mathbb{C})$ to 
$F(\mathcal{W})^{\alpha}$, \ \\
(ii) if $x$ and $y$ are normal elements in $F(\mathcal{W})^{\alpha}$ such that 
$\mathrm{Sp}(x)=\mathrm{Sp}(y)$ and 
$0<\tau_{\mathcal{W}, \omega} (f(x))=\tau_{\mathcal{W}, \omega}(f(y))$ 
for any $f\in C(\mathrm{Sp}(x))_{+}\setminus\{0\}$, then $x$ and $y$ are unitary equivalent 
in $F(\mathcal{W})^{\alpha}$, \ \\
(iii) there exists an injective $*$-homomorphism from 
$\mathcal{W}\rtimes_{\alpha}\Gamma$ to $\mathcal{W}$.  
\end{mainthm2}

We use a first cohomology vanishing type theorem (Corollary \ref{cor:vanishing}) for showing 
that if $\Gamma\in \mathfrak{C}$ and $\alpha$ is a strongly outer action of $\Gamma$ on 
$\mathcal{W}$, then $F(\mathcal{W})^{\alpha}$ satisfies the properties (i) and (ii) 
in the theorem above. 
Kishimoto's techniques for Rohlin type theorems in \cite{Kis0} and \cite{Kis00}, 
Herman-Ocneanu's argument in \cite{HO} and homotopy type arguments in \cite{Na2} 
enable us to show this first cohomology vanishing type theorem. 
Also, note that our arguments for $F(\mathcal{W})^{\alpha}$ are based on results that are shown 
by techniques around (equivariant) property (SI) in \cite{MS}, \cite{MS2}, \cite{MS3}, 
\cite{Sa0}, \cite{Sa}, \cite{Sa2} and \cite{Sza6}.

\section{Preliminaries}\label{sec:pre}

\subsection{Notations and basic definitions}

Let $\alpha$ and $\beta$ be actions of a countable discrete group $\Gamma$ 
on C$^*$-algebras $A$ and $B$, respectively. 
We say that $\alpha$ is \textit{conjugate to} $\beta$ if there exists a isomorphism
$\varphi$ from $A$ onto $B$ such that $\varphi\circ \alpha_g=\beta_g\circ \varphi$ for any 
$g\in \Gamma$. 
Note that $\alpha$ induces an action on the multiplier algebra $M(A)$ of $A$. We denote it by the 
same symbol $\alpha$. 
An $\alpha$-cocycle on $A$ is a map from $\Gamma$ to the unitary group of $M(A)$ such that 
$u_{gh}=u_{g}\alpha_g(u_h)$ for any $g,h\in \Gamma$. 
We say that $\alpha$ is \textit{cocycle conjugate to} $\beta$ if there exist an 
isomorphism $\varphi$ from $A$ onto $B$ and a $\beta$-cocycle $u$ such that 
$\varphi\circ \alpha_g=\mathrm{Ad}(u_g)\circ \beta_g \circ \varphi$ for any $g\in\Gamma$. 
An action $\alpha$ of $\Gamma$ on $A$ is said to be \textit{outer} if $\alpha_g$ is not an 
inner automorphism of $A$ for any $g\in \Gamma\setminus \{\iota\}$ 
where $\iota$ is the identity of $\Gamma$. 
We denote by $A^{\alpha}$ and $A\rtimes_{\alpha}\Gamma$ the fixed point algebra 
and the reduced crossed product C$^*$-algebra, respectively.  

Assume that $A$ has a unique tracial state $\tau_{A}$. 
Let $(\pi_{\tau_{A}}, H_{\tau_{A}})$ be the Gelfand-Naimark-Segal (GNS) representation of $\tau_{A}$. 
Then $\pi_{\tau_{A}}(A)^{''}$ is a finite factor and $\alpha$ induces an action $\tilde{\alpha}$ on 
$\pi_{\tau_{A}}(A)^{''}$. We say that $\alpha$ is \textit{strongly outer} if 
$\tilde{\alpha}$ is an outer action on $\pi_{\tau_{A}}(A)^{''}$. 
(We refer the reader to \cite{GHV} and \cite{MS2} for the definition of strongly outerness 
for more general settings.)

We denote by $\mathcal{R}_0$ and $M_{2^{\infty}}$ the injective II$_1$ factor and the CAR 
algebra, respectively. 

\subsection{Fixed point algebras of Kirchberg's central sequence C$^*$-algebras}\label{sec:fixed}

Let $\omega$ be a free ultrafilter on $\mathbb{N}$, and put 
$$
A^{\omega}:= \ell^{\infty}(\mathbb{N}, A)/\{\{x_n\}_{n\in\mathbb{N}}\in \ell^{\infty}(\mathbb{N}, A)\; |
\; \lim_{n\to\omega} \|x_n \|=0 \}.
$$
We denote by $(x_n)_n$ a representative of an element in $A^{\omega}$. 
We identify $A$ with the C$^*$-subalgebra of $A^{\omega}$ consisting of equivalence classes of 
constant sequences. 
Set
$$
\mathrm{Ann}(A, A^{\omega}):= \{(x_n)_n\in A^{\omega}\cap A^{\prime}\; |\; \lim_{n\to \omega}\| x_na\|=0 
\text{ for any } a\in A\}.
$$
Then $\mathrm{Ann}(A, A^{\omega})$ is an closed ideal of $A^{\omega}\cap A^{\prime}$, and define 
$$
F(A):= A^{\omega}\cap A^{\prime}/\mathrm{Ann}(A, A^{\omega}).
$$
See \cite{Kir2} for basic properties of $F(A)$. For a finite von Neumann algebra $M$, put
$$
M^{\omega}:=\ell^{\infty}(\mathbb{N}, M)/\{\{x_n\}_{n\in\mathbb{N}}\in \ell^{\infty}(\mathbb{N}, M)
\; |\; \lim_{n\to\omega}\| x_n\|_2=0\}
$$
and
$$ 
M_{\omega}:=M^{\omega}\cap M^{\prime}. 
$$
Note that we identify $M$ with the subalgebra of $M^{\omega}$ consisting of equivalence classes of 
constant sequences and $M_{\omega}$ is the von Neumann algebraic central sequence algebra 
(or the asymptotic centralizer) of $M$.  

For a tracial state $\tau_A$ on $A$, define a map $\tau_{A, \omega}$ from $F(A)$ to $\mathbb{C}$ by 
$\tau_{A, \omega}([(x_n)_n])=\lim_{n\to\omega}\tau_A(x_n)$ for any $[(x_n)_n]\in  F(A)$. 
Then $\tau_{A, \omega}$ is a well defined tracial state on $F(A)$ by \cite[Proposition 2.1]{Na2}.
Put $J_{\tau_{A}}:=\{x\in F(A)\; |\; \tau_{A, \omega}(x^*x)=0\}$. 
If $A$ is separable and $\tau_{A}$ is faithful, then $\pi_{\tau_{A}}$ induces an isomorphism 
from $F(A)/J_{\tau_{A}}$ onto $\pi_{\tau_{A}}(A)^{''}_{\omega}$ by essentially the same argument as in 
the proof of \cite[Theorem 3.3]{KR}. 
In this paper, the reindexing argument and the diagonal argument (or Kirchberg's $\varepsilon$-test 
\cite[Lemma A.1]{Kir2}) are frequently used. 
We refer the reader to \cite[Section 1.3]{BBSTWW} and \cite[Chapter 5]{Oc} for details 
of these arguments. 
Every action $\alpha$ of a countable discrete group on $A$ induces an action on $F(A)$. 
We denote it by the same symbol $\alpha$ for simplicity. 
Note that if $\alpha$ on $A$ are cocycle conjugate to $\beta$ on $B$, then $\alpha$ on $F(A)$ 
are conjugate to $\beta$ on $F(B)$. 
If $A$ is simple, separable and monotracial, then $\tilde{\alpha}$ also induces an action on 
$\pi_{\tau_{A}}(A)^{''}_{\omega}$. We also denote it by the same symbol $\tilde{\alpha}$. 
By \cite[Proposition 3.9]{Na5}, we see that $\pi_{\tau_{A}}$ induces an 
isomorphism from $F(A)^{\alpha}/J_{\tau_{A}}^{\alpha}$ onto 
$(\pi_{\tau_{A}}(A)^{''})_{\omega}^{\tilde{\alpha}}$. 

The following proposition is an immediate consequence of \cite[Theorem 3.6]{Na5}, 
\cite[Proposition 3.11]{Na5} and \cite[Proposition 3.12]{Na5}. 
Note that these propositions are based on results in \cite{MS}, \cite{MS2}, \cite{MS3}, 
\cite{Sa0}, \cite{Sa}, \cite{Sa2} and \cite{Sza6}. 

\begin{pro}\label{pro:MS-pro}
Let $\alpha$ be an outer action of a countable discrete amenable group on $\mathcal{W}$. 
\ \\
(1) The Razak-Jacelon algebra $\mathcal{W}$ has property (SI) relative to $\alpha$, 
that is, if $a$ and $b$ are positive contractions in $F(\mathcal{W})^{\alpha}$ 
satisfying $\tau_{\mathcal{W}, \omega}(a)=0$ and 
$\inf_{m\in\mathbb{N}}\tau_{\mathcal{W}, \omega}(b^m)>0$, then there exists an element $s$ in 
$F(\mathcal{W})^{\alpha}$ such that $bs=s$ and $s^*s=a$. 
\ \\
(2) The fixed point algebra $F(\mathcal{W})^{\alpha}$ is monotracial.  
\ \\
(3) If $a$ and $b$ are 
positive elements in $F(\mathcal{W})^{\alpha}$ 
satisfying $d_{\tau_{\mathcal{W}, \omega}}(a)< d_{\tau_{\mathcal{W}, \omega}}(b)$, then 
there exists an element $r$ in $F(\mathcal{W})^{\alpha}$ such that $r^*br=a$. 
\end{pro}

\begin{Def}\label{def:property-w}
Let $\alpha$ be an action of a countable discrete group $\Gamma$ on $\mathcal{W}$. We say that 
$\alpha$ has \textit{property W} if  $\alpha$ satisfies the following properties: \ \\
(i) there exists a unital  $*$-homomorphism from $M_{2}(\mathbb{C})$ to 
$F(\mathcal{W})^{\alpha}$, \ \\
(ii) if $x$ and $y$ are normal elements in $F(\mathcal{W})^{\alpha}$ such that 
$\mathrm{Sp}(x)=\mathrm{Sp}(y)$ and 
$0<\tau_{\mathcal{W}, \omega} (f(x))=\tau_{\mathcal{W}, \omega}(f(y))$ 
for any $f\in C(\mathrm{Sp}(x))_{+}\setminus\{0\}$, then $x$ and $y$ are unitary equivalent 
in $F(\mathcal{W})^{\alpha}$. 
\end{Def}

Note that if there exists a unital $*$-homomorphism from $M_{2}(\mathbb{C})$ to 
$F(\mathcal{W})^{\alpha}$, then $\alpha$ on $\mathcal{W}$ 
is cocycle conjugate to $\alpha\otimes\mathrm{id}_{M_{2^{\infty}}}$ on 
$\mathcal{W}\otimes M_{2^{\infty}}$. 
Indeed, there exists a unital $*$-homomorphism from $M_{2^{\infty}}$ to $F(\mathcal{W})^{\alpha}$ by 
a similar argument as \cite[Corollary 1.13]{Kir2}. 
Hence \cite[Corollary 3.8]{Sza1} (see also \cite{Sza1c}) implies this cocycle conjugacy.  
Using this observation, Proposition \ref{pro:MS-pro} and Definition \ref{def:property-w} 
instead of $M_{2^{\infty}}$-stability of $\mathcal{W}$, \cite[Proposition 4.1]{Na2}, \cite[Theorem 5.3]{Na2}
and \cite[Theorem 5.8]{Na2},
we obtain the following theorem by essentially the same arguments as in the proofs of 
\cite[Proposition 4.2]{Na2}, \cite[Theorem 5.7]{Na2} and 
\cite[Corollary 5.11]{Na2}(or \cite[Corollary 5.5]{Na3}). 

\begin{thm}\label{thm:homotopy}
Let $\alpha$ be an outer action of a countable discrete amenable group on $\mathcal{W}$. 
Assume that $\alpha$ has property W.   \ \\
(1) For any $\theta\in [0,1]$, there exists a projection $p$ in $F(\mathcal{W})^{\alpha}$ 
such that $\tau_{\mathcal{W}, \omega}(p)=\theta$. \ \\
(2) For any unitary element $u$ in $F(\mathcal{W})^{\alpha}$, there exists a continuous 
path of unitaries $U: [0,1]\to F(\mathcal{W})^{\alpha}$ such that 
$$
U(0)=1,\quad  U(1)=u \quad \text{and} \quad \mathrm{Lip}(U)\leq 2\pi 
$$ 
where $\mathrm{Lip}(U)$ is the Lipschitz constant of $U$, that is, the smallest positive number satisfying 
$
\| U(t)- U(s)\| \leq \mathrm{Lip}(U)|t-s|
$ 
for any $t, s\in [0,1]$. 
\ \\
(3) If $p$ and $q$ are projections in $F(\mathcal{W})^{\alpha}$ such that 
$0<\tau_{\mathcal{W}, \omega}(p)=\tau_{\mathcal{W}, \omega}(q)$, then $p$ and $q$ are 
Murray-von Neumann equivalent. 
\end{thm}

For any countable discrete group $\Gamma$, let $\mu^{\Gamma}$ be the Bernoulli shift action of 
$\Gamma$ on $\bigotimes_{g\in \Gamma}M_{2^{\infty}}\cong M_{2^{\infty}}$. 
The following theorem is a slight variant of \cite[Theorem 8.1]{Na5}. 
Note that one of the main techniques in the proof of \cite[Theorem 8.1]{Na5} is 
Szab\'o's approximate cocycle intertwining argument \cite{Sza} (see also \cite{Ell2}).  

\begin{thm}\label{thm:8.1}
Let $\alpha$ be a strongly outer action of a countable discrete amenable group $\Gamma$ on 
$\mathcal{W}$. Then $\alpha$ is cocycle conjugate to 
$\mu^{\Gamma}\otimes\mathrm{id}_{\mathcal{W}}$ on $M_{2^{\infty}}\otimes \mathcal{W}$ 
if and only if $\alpha$ has property W and there exists an injective $*$-homomorphism from 
$\mathcal{W}\rtimes_{\alpha}\Gamma$ to $\mathcal{W}$.  
\end{thm}
\begin{proof}
\cite[Proposition 4.2]{Na5}, \cite[Theorem 4.5]{Na5} and \cite[Theorem 8.1]{Na5} imply the only if part. 
The if part is an immediate consequence of \cite[Theorem 8.1]{Na5} and Theorem \ref{thm:homotopy}. 
\end{proof}

\section{First cohomology vanishing type theorem}

In this section, we shall show a first cohomology vanishing type theorem (Corollary \ref{cor:vanishing}). 
This is a corollary of a Rohlin type theorem (Theorem \ref{thm:Rohlin}). 

The following lemma is well known among experts. See, for example, 
\cite[Theorem 4.8]{Ka} for a similar (but not the same) result. 
For the reader's convenience, we shall give a proof based on Ocneanu's classification theorem 
\cite[Corollary 1.4]{Oc}. 

\begin{lem}\label{lem:outer}
Let $\Gamma$ be a countable discrete amenable group, and let $N$ be a normal subgroup 
of $\Gamma$. If 
$\gamma$ is an outer action of $\Gamma$ on the injective II$_1$ factor $\mathcal{R}_0$ and 
$g_0\notin N$, then $\gamma_{g_0}$ induces a properly outer automorphism of 
$(\mathcal{R}_0)_{\omega}^{\gamma|_N}$. 
\end{lem}
\begin{proof}
Since $N$ is a normal subgroup, it is clear that $\gamma_{g_0}$ induces an automorphism 
of $(\mathcal{R}_0)_{\omega}^{\gamma|_N}$. 
First, we shall show that $\gamma_{g_0}$ is not trivial as an automorphism of 
$(\mathcal{R}_0)_{\omega}^{\gamma|_N}$. 
Let $\pi$ be the quotient map from $\Gamma$ to $\Gamma/ N$, and let $\beta$ be the 
Bernoulli shift action of $\Gamma /N$ on 
$\mathcal{R}_0\cong \bigotimes_{\pi (g)\in \Gamma /N}\mathcal{R}_0$. 
Define an action $\delta$ of $\Gamma$ on 
$\mathcal{R}_0\cong \mathcal{R}_0 \bar{\otimes}\mathcal{R}_{0}$ by 
$\delta_g:= \gamma_g\otimes \beta_{\pi (g)}$ for any $g\in\Gamma$. 
By Ocneanu's classification theorem \cite[Corollary 1.4]{Oc}, $\gamma$ on $\mathcal{R}_0$ and 
$\delta$ on $\mathcal{R}_0 \bar{\otimes}\mathcal{R}_{0}$ are cocycle conjugate. Hence 
there exists an isomorphism $\Phi$ from $(\mathcal{R}_0)_{\omega}$ onto 
$(\mathcal{R}_0 \bar{\otimes}\mathcal{R}_{0})_{\omega}$ such that 
$\Phi\circ \gamma_g=\delta_g\circ \Phi$ for any $g\in \Gamma$. 
Since $\beta_{\pi (g_0)}$ is an outer automorphism of $\mathcal{R}_0$, there exists an element $(x_n)_n$ 
in $(\mathcal{R}_0)_{\omega}$ such that $(\beta_{\pi(g_0)}(x_n))_n\neq (x_n)_n$ by \cite[Theorem 3.2]{C3}. 
Put $(y_n)_n:=\Phi^{-1}((1\otimes x_n)_n)\in (\mathcal{R}_0)_{\omega}$. 
Then it is easy to see that we have $(y_n)_n\in (\mathcal{R}_0)_{\omega}^{\gamma|_N}$ and 
$(\gamma_{g_0}(y_n))_n\neq (y_n)_n$. Finally, we shall show that $\gamma_{g_0}$ is properly outer 
as an automorphism of $(\mathcal{R}_0)_{\omega}^{\gamma|_N}$. 
Since $(\mathcal{R}_0)_{\omega}^{\gamma|_N}$ is a factor (see, for example, \cite[Lemma 4.1]{MS2}), 
it is enough to show that $\gamma_{g_0}$ is outer as an automorphism of 
$(\mathcal{R}_0)_{\omega}^{\gamma|_N}$. 
In particular, we shall show that for any element $(u_n)_n$ in $(\mathcal{R}_0)_{\omega}^{\gamma|_N}$, 
there exists an element $(z_n)_n$ in $(\mathcal{R}_0)_{\omega}^{\gamma|_N}$ such that 
$(u_nz_n)_n=(z_nu_n)_n$ and $(\gamma_{g_0}(z_n))_n\neq (z_n)_n$. 
Taking a suitable subsequence of $(y_n)_n$ (or the reindexing argument), we obtain the desired element 
$(z_n)_n$. Consequently, the proof is complete. 
\end{proof}

Consider a semidirect product group $N\rtimes\mathbb{Z}$. 
For $g\in N$ and $m\in\mathbb{Z}$, 
let $(g, m)$ denote an element in $N\rtimes \mathbb{Z}$. Note that we have 
$N\rtimes \mathbb{Z}=\{(g,m)\; |\; g\in N, m\in\mathbb{Z}\}$.
The following lemma is an analogous lemma of \cite[Lemma 6.2]{Na2}. 
See also \cite[Theorem 3.4]{MS1}. 

\begin{lem}\label{lem:6.2}
Let $\Gamma$ be a semidirect product $N\rtimes \mathbb{Z}$ where $N$ is a countable discrete 
amenable group, and let $\alpha$ be a strongly outer action of $\Gamma$ on $\mathcal{W}$. 
Then for any $k\in\mathbb{N}$, there exists a positive contraction $f$ in 
$F(\mathcal{W})^{\alpha|_{N}}$ such 
that 
$$
\tau_{\mathcal{W}, \omega}(f)=\frac{1}{k}\quad \text{and} \quad f\alpha_{(\iota,j)}(f)=0
$$
for any $1\leq j \leq k-1$. 
\end{lem}
\begin{proof} 
Since $\pi_{\tau_{\mathcal{W}}}(\mathcal{W})^{''}$ is isomorphic to the injective II$_1$ factor, 
Lemma \ref{lem:outer} implies that $\tilde{\alpha}_{(\iota, 1)}$ is an aperiodic automorphism of 
$(\pi_{\tau_{\mathcal{W}}}(\mathcal{W})^{''})_{\omega}^{\tilde{\alpha}|_N}$. 
Hence it follows from \cite[Theorem 1.2.5]{C2} that 
there exists a partition of unity $\{P_{j}\}_{j=1}^k$ consisting of projections in 
$(\pi_{\tau_{\mathcal{W}}}(\mathcal{W})^{''})_{\omega}^{\tilde{\alpha}|_N}$ such that 
$\tilde{\alpha}_{(\iota, 1)}(P_{j})=P_{j+1}$ for any $1\leq j \leq k-1$. 
Since $\pi_{\tau_{\mathcal{W}}}$ induces an isomorphism from 
$F(\mathcal{W})^{\alpha|_N}/J_{\tau_{\mathcal{W}}}^{\alpha|_N}$ onto 
$(\pi_{\tau_{\mathcal{W}}}(\mathcal{W})^{''})_{\omega}^{\tilde{\alpha}|_N}$ (see Section \ref{sec:fixed}), 
there exists a positive contraction $[(e_n)_n]$ in $F(\mathcal{W})^{\alpha|_N}$ such that 
$(\pi_{\tau_{\mathcal{W}}}(e_n))_n=P_1$ in 
$(\pi_{\tau_{\mathcal{W}}}(\mathcal{W})^{''})_{\omega}^{\tilde{\alpha}|_N}$. 
Then we have 
$$
\lim_{n\to\omega} \| \pi_{\tau_{\mathcal{W}}}(e_n \alpha_{(\iota ,j)}(e_n))\|_2=0
\quad 
\text{and}
\quad
\tau_{\mathcal{W}, \omega}([(e_n)_n])=\tilde{\tau}_{\mathcal{W}, \omega}(P_1)=\dfrac{1}{k}
$$
for any $1\leq j \leq k-1$, 
where $\tilde{\tau}_{\mathcal{W}, \omega}$ is the induced tracial state on 
$(\pi_{\tau_{\mathcal{W}}}(\mathcal{W})^{''})_{\omega}^{\tilde{\alpha}|_N}$ by $\tau_{\mathcal{W}}$. 
The rest of the proof is the same as \cite[Lemma 6.2]{Na2}. (See also \cite[Proposition 3.3]{MS1}.) 
\end{proof}

Using Proposition \ref{pro:MS-pro}, Theorem \ref{thm:homotopy} (we need to assume that 
$\alpha|_N$ has property W) and Lemma \ref{lem:6.2} instead of 
\cite[Proposition 4.1]{Na2}, \cite[Proposition 4.2]{Na2}, \cite[Theorem 5.8]{Na2} and 
\cite[Lemma 6.2]{Na2}, 
we obtain the following Rohlin type theorem by essentially the same arguments in the proofs of 
\cite[Lemma 6.3]{Na2} and \cite[Theorem 6.4]{Na2}. 
Note that these arguments are based on \cite{Kis0} and \cite{Kis00}.

\begin{thm}\label{thm:Rohlin}
Let $\Gamma$ be a semidirect product $N\rtimes \mathbb{Z}$ where $N$ is a countable discrete 
amenable group, and let $\alpha$ be a strongly outer action of $\Gamma$ on $\mathcal{W}$. 
Assume that $\alpha|_{N}$ has property W. 
Then for any $k\in\mathbb{N}$, there exists a partition on unity $\{p_{1,i}\}_{i=0}^{k-1}\cup 
\{p_{2,j}\}_{j=0}^{k}$ consisting of projections in $F(\mathcal{W})^{\alpha|_N}$ such that 
$$
\alpha_{(\iota, 1)}(p_{1,i})=p_{1, i+1}\quad \text{and} \quad \alpha_{(\iota, 1)}(p_{2,j})=p_{2,j+1}
$$
for any $0\leq i\leq k-2$ and $0\leq j\leq k-1$. 
\end{thm}

Theorem \ref{thm:homotopy}, Theorem \ref{thm:Rohlin} and Herman-Ocneanu's argument 
\cite[Theorem 1]{HO} (see also remarks after \cite[Lemma 1]{HO}, \cite{I0} and \cite{Kis3}) 
imply the following corollary. 

\begin{cor}\label{cor:vanishing}
Let $\Gamma$ be a semidirect product $N\rtimes \mathbb{Z}$ where $N$ is a countable discrete 
amenable group, and let $\alpha$ be a strongly outer action of $\Gamma$ on $\mathcal{W}$. 
Assume that $\alpha|_{N}$ has property W and $S$ is a countable subset in $F(\mathcal{W})^{\alpha|_N}$. 
For any unitary element $u$ in $F(\mathcal{W})^{\alpha|_N}\cap S^{\prime}$, 
there exists a unitary element $v$ in $F(\mathcal{W})^{\alpha|_N}\cap S^{\prime}$ 
such that $u=v\alpha_{(\iota, 1)}(v)^{*}$. 
\end{cor}

\section{Main theorem}

In this section we shall show the main theorem. 
Recall that $\mathfrak{C}$ is the smallest class of countable discrete groups with the 
following properties:(i) $\mathfrak{C}$ contains the trivial group, (ii) $\mathfrak{C}$ is 
closed under isomorphisms, countable increasing unions and extensions by $\mathbb{Z}$. 
Note that if $\Gamma$ is an extension of $N$ by $\mathbb{Z}$, then $\Gamma$ is isomorphic to 
a semidirect product $N\rtimes\mathbb{Z}$.  

The following lemma is an easy consequence of the definition of 
property W and the diagonal argument. 
\begin{lem}\label{lem:unions}
Let $\Gamma$ be an increasing union $\bigcup_{m\in\mathbb{N}}\Gamma_m$ of discrete 
countable groups $\Gamma_m$, and let  $\alpha$ be an action of $\Gamma$ on $\mathcal{W}$. 
If $\alpha|_{\Gamma_m}$ has property W for any $m\in\mathbb{N}$, then $\alpha$ has property W. 
\end{lem}

The following lemma is an application of Corollary \ref{cor:vanishing}. 

\begin{lem}\label{lem:extension-z}
Let $\Gamma$ be a semidirect product $N\rtimes \mathbb{Z}$ where $N$ is a countable discrete 
amenable group, and let $\alpha$ be a strongly outer action of $\Gamma$ on $\mathcal{W}$. 
If $\alpha|_{N}$ has property W, then $\alpha$ has property W. 
\end{lem}
\begin{proof}
(i) There exists a unital 
$*$-homomorphism $\varphi$ from $M_{2}(\mathbb{C})$ to $F(\mathcal{W})^{\alpha|_N}$ by the 
assumption. Let $\{e_{ij}\}_{i,j=1}^{2}$ be the standard matrix units of $M_{2}(\mathbb{C})$. 
Since we have $0<\tau_{\mathcal{W}, \omega}(\varphi(e_{11}))
=\tau_{\mathcal{W}, \omega}(\alpha_{(\iota,1)}(\varphi(e_{11})))$, there exists an 
element $w$ in  $F(\mathcal{W})^{\alpha|_N}$ such that $w^*w
=\alpha_{(\iota, 1)}(\varphi(e_{11}))$ and $ww^*=\varphi(e_{11})$ 
by Theorem \ref{thm:homotopy}. 
Put $u:= \sum_{i=1}^{2}\varphi(e_{i1})w\alpha_{(\iota ,1)}(\varphi(e_{1i}))$. 
Then $u$ is a unitary element in $F(\mathcal{W})^{\alpha|_N}$ such that 
$\alpha_{(\iota, 1)}(\varphi(x))=u^*\varphi(x)u$ for any $x\in M_{2}(\mathbb{C})$. 
By Corollary \ref{cor:vanishing}, there exists a unitary element $v$ in $F(\mathcal{W})^{\alpha|_N}$ 
such that $u=v\alpha_{(\iota ,1)}(v)^*$. 
We have $\alpha_{(\iota, 1)}(v^*\varphi(x)v)= v^*\varphi(x)v$ for any $x\in M_{2}(\mathbb{C})$. 
Hence the map $\psi$ defined by $\psi(x):=v^*\varphi(x)v$ for any 
$x\in M_{2}(\mathbb{C})$ is a unital $*$-homomorphism from $M_{2}(\mathbb{C})$ to 
$F(\mathcal{W})^{\alpha}$. 
(ii) Let $x$ and $y$ be normal elements in $F(\mathcal{W})^{\alpha}$ such that 
$\mathrm{Sp}(x)=\mathrm{Sp}(y)$ and 
$0<\tau_{\mathcal{W}, \omega} (f(x))=\tau_{\mathcal{W}, \omega}(f(y))$
for any $f\in C(\mathrm{Sp}(x))_{+}\setminus\{0\}$. Since $x$ and $y$ are also elements in 
$F(\mathcal{W})^{\alpha|_N}$, there exists a unitary element $u$ in $F(\mathcal{W})^{\alpha|_N}$ such 
that $uxu^*=y$ by the assumption. Note that $u\alpha_{(\iota, 1)}(u)^*$ is a unitary element in 
$F(\mathcal{W})^{\alpha|_N}\cap \{ y\}^{\prime}$. Hence Corollary \ref{cor:vanishing} implies that 
there exists a unitary element $v$ in $F(\mathcal{W})^{\alpha|_N}\cap \{ y\}^{\prime}$ such that 
$u\alpha_{(\iota, 1)}(u)^*=v\alpha_{(\iota, 1)}(v)^*$. 
We have $\alpha_{(\iota ,1)}(v^*u)=v^*u$ and $v^*uxu^*v= v^*yv=y$. 
Therefore $x$ and $y$ are unitary equivalent in $F(\mathcal{W})^{\alpha}$. 
By (i) and (ii), $\alpha$ has property W. 
\end{proof}

The following theorem is the main theorem in this paper. 

\begin{thm}\label{thm:main}
Let $\Gamma$ be a countable discrete group in $\mathfrak{C}$, and let $\alpha$ be a strongly outer 
action of $\Gamma$ on $\mathcal{W}$. Then $\alpha$ is cocycle conjugate to 
$\mu^{\Gamma}\otimes \mathrm{id}_{\mathcal{W}}$ on $M_{2^{\infty}}\otimes\mathcal{W}$. 
\end{thm}
\begin{proof}
Every action of the trivial group on $\mathcal{W}$ has property W by results in 
\cite{Na2} (or \cite[Theorem 3.8]{Na4}). 
By Lemma \ref{lem:unions} and Lemma \ref{lem:extension-z}, 
we see that $\alpha$ has property W. 
Note that this implies that $\mathcal{W}\rtimes_{\alpha}\Gamma$ is $M_{2^{\infty}}$-stable 
because $\alpha$ is cocycle conjugate to $\alpha\otimes\mathrm{id}_{M_{2^{\infty}}}$ on 
$\mathcal{W}\otimes M_{2^{\infty}}$. 
Since the class of separable nuclear C$^*$-algebras that are $KK$-equivalent to $\{0\}$ is closed 
under countable inductive limits and crossed products by $\mathbb{Z}$, 
\cite[Theorem 6.1]{Na4} implies that $\mathcal{W}\rtimes_{\alpha}\Gamma$ is isomorphic to 
$\mathcal{W}$. 
Therefore we obtain the conclusion by Theorem \ref{thm:8.1}.
\end{proof}

The following corollary is an immediate consequence of the theorem above. 

\begin{cor}
Let $\Gamma$ be a countable discrete group in $\mathfrak{C}$. Then there exists a unique strongly 
outer action of $\Gamma$ on $\mathcal{W}$ up to cocycle conjugacy. 
\end{cor}

\end{document}